\documentclass[10pt]{article}

\usepackage{amsmath, amsthm, amssymb}
\usepackage{caption,color, graphicx, subcaption}
\usepackage[text={6.5in,9.25in},centering]{geometry}
\usepackage{paralist}
\usepackage[sort, numbers]{natbib}
\usepackage{url}

\setcaptionmargin{0.25in}

\def\argmin{\mathop\mathrm{argmin}}

\newcommand{\norm}[1]{\left\lVert#1\right\rVert}

\newcommand{\rmD}{\mathrm{D}}
\newcommand{\rmd}{\mathrm{d}}

\theoremstyle{definition}
\newtheorem{theorem}{Theorem}[section]
\makeatletter\@addtoreset{equation}{section}\makeatother

\newcommand{\revised}[1]{#1}

\begin{document}

\title{Enabling equation-free modeling via diffusion maps}

\author{%
Tracy Chin\thanks{Department of Mathematics, University of Washington, Seattle, USA} \and
Jacob Ruth\thanks{Unaffiliated} \and
Clayton Sanford\thanks{Department of Computer Science, Columbia University, New York City, USA} \and
Rebecca Santorella\thanks{Division of Applied Mathematics, Brown University, Providence, USA} \and
Paul Carter\thanks{Department of Mathematics, University of California, Irvine, USA} \and
Bj\"orn Sandstede\footnotemark[4]}

\date{\today}
\maketitle

\begin{abstract}
Equation-free modeling aims at extracting low-dimensional macroscopic dynamics from complex high-dimensional systems that govern the evolution of microscopic states. This algorithm relies on lifting and restriction operators that map macroscopic states to microscopic states and vice versa. Combined with simulations of the microscopic state, this algorithm can be used to apply Newton solvers to the implicitly defined low-dimensional macroscopic system or solve it more efficiently using direct numerical simulations. The key challenge is the construction of the lifting and restrictions operators that usually require a priori insight into the underlying application. In this paper, we design an application-independent algorithm that uses diffusion maps to construct these operators from simulation data. Code is available at \url{https://doi.org/10.5281/zenodo.5793299}.
\end{abstract}


\section{Introduction}
\label{introduction}

In many complex dynamical systems, low-dimensional macroscopic behavior emerges from interactions at the high-dimensional microscopic level. For instance, traffic jams are global macroscopic structures that emerge from the interactions of many individual cars that move along a road: traffic jams can be captured meaningfully by a single macroscopic quantity, namely the standard deviation of the distances between consecutive cars from the mean (see below for further details). Mathematically, these macroscopic structures live on low-dimensional invariant manifolds, and our goal is to exploit their existence, and the accompanying reduction in effective dimension, when we conduct bifurcation analyses or carry out direct simulations. \revised{In certain cases, we may be able to characterize these invariant manifolds explicitly, either because they are given as graphs of explicit functions or because the microscopic variables decouple from the macroscopic system. In most cases, however, these invariant manifolds are not known. We are interested in the latter case,  particularly where the microscopic variables are interchangeable when considering macroscopic effects and structures.}

Equation-free modeling, so named because the macroscopic system is not governed by an explicit ordinary differential equation, estimates macroscopic behavior through a multi-scale approach that exploits the connection between the macro and microlevels \cite{Kevrekidis2004, Kevrekidis2009, Kevrekidis2003}. All equation-free methods depend on the following algorithm that attempts to describe the macroscopic system implicitly \cite{Kevrekidis2009}:
\begin{compactenum}[(1)]
\item lift: build the microstate from the macrostate using the lifting operator $\mathcal{L}$;
\item evolve: simulate the microstate for short bursts using the evolution $\Phi_t$; and 
\item restrict: calculate the macrostate from the evolved microstate using the restriction operator $\mathcal{R}$.
\end{compactenum}
One of the biggest challenges in equation-free modeling is the selection of lifting and restriction operators since the choice of macroscopic observables may not always be obvious \cite{MarschlerPRE}. One way to pick relevant macroscopic variables is to use a dimension-reduction technique such as diffusion maps. Diffusion maps embed high dimensional data sets into low-dimensional Euclidean spaces. Unlike standard linear methods such as principal component analysis, diffusion maps are able to find useful low-dimensional parametrizations of the original data sets even when the data lie on or near a nonlinear manifold \cite{Coifman2006, Coifman2005, lafon2004}. In this paper, we show that low-dimensional embeddings given through diffusion maps can be used to identify macroscopic variables in equation-free modeling and define lifting and restriction operators. \revised{We note that this approach will not necessarily provide a physical interpretation of the resulting parametrization, though it will often result in macroscopic variables that are physically relevant: We refer to \S\ref{s:embedding} below, \cite[Inset in Figure~5]{Frewen}, and \cite[\S~IV.A and Figure~5]{Sonday2009} for examples, and to \cite{Meila} for an algorithm that interprets macroscopic variables in terms of a prescribed list of relevant physical quantities.}

\begin{figure}
\centering
\includegraphics{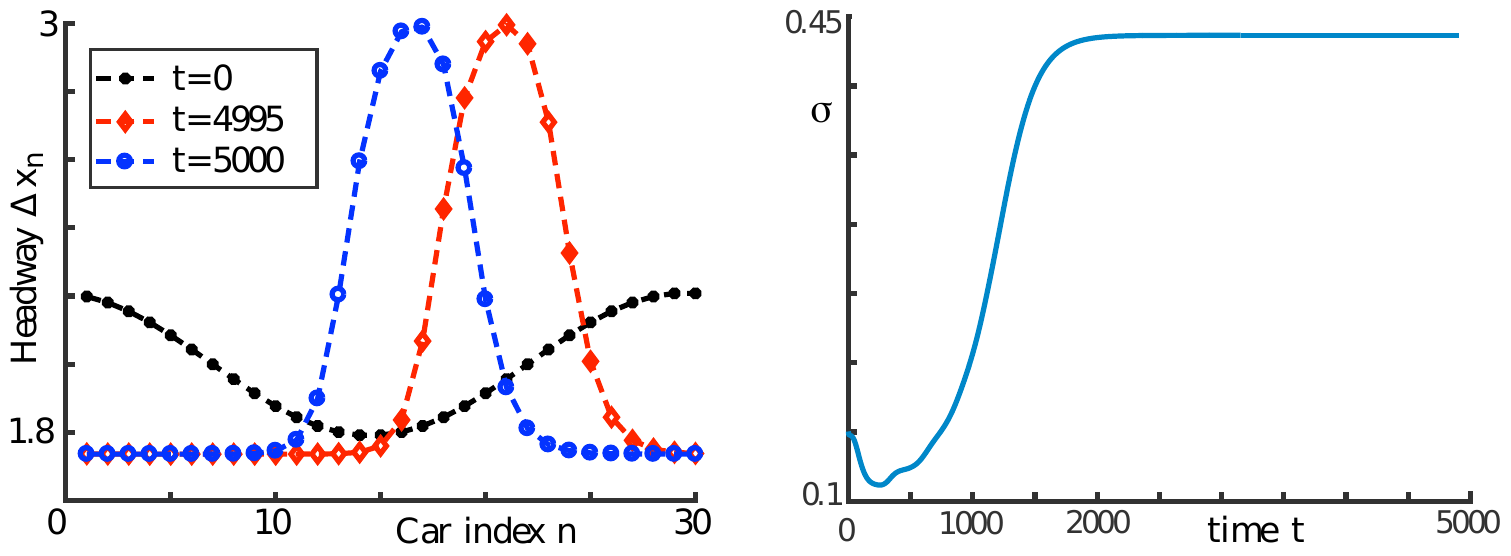}
\caption{Shown are the time evolution of the headway profile (left) and the standard deviation from the mean headway (right) for a solution of (\ref{trafficEq}) with $v_0=1$ (all other parameters are as in Table~\ref{tab:param}). In the left panel, the initial state (black) is compared with the final traveling-wave solution (red to blue): note that the traffic jam corresponds to the region where the headways are small, so that the car density is high. The right panel shows the time evolution of the deviation $\sigma$ of the headways from the mean headway: the increase and eventual convergence of $\sigma$ to a larger value indicates the emergence of a stable traffic jam.}\label{trafficSim}
\end{figure}

To illustrate our algorithm and demonstrate its effectiveness, we apply it to the same traffic model \cite{bando1995dynamical} to which traditional equation-free modeling had been applied in earlier work \cite{MarschlerSIAM}. In this model, $N$ cars drive around a ring road of length $L$. We assume that all drivers follow the same deterministic behavior governed by
\begin{equation}
    \tau \frac{\rmd^2x_n}{\rmd t^2} + \frac{\rmd x_n}{\rmd t}  = V(x_{n+1} - x_n), \quad x_n\in \mathbb{R}/L\mathbb{Z}, \quad n = 1, 2, ... N,
    \label{trafficEq}
\end{equation}
where $x_n$ is the position of the $n$th car, $\tau$ reflects the inertia of cars (or, alternatively, reaction times of drivers), and $V$ is the optimal velocity function defined by
\[
V(d) = v_0(\tanh(d-h) + \tanh(h)),
\]
where $v_0(1+\tanh(h))$ is the maximal velocity and $h$ determines the desired safety distance between cars. In this model, each driver adjusts their acceleration to attain an optimal velocity based on the distance $\Delta x_n = x_{n+1} - x_n$ to the car in front, which is also referred to as the headway. Two common traffic patterns can emerge as stable solutions in this model, namely free-flow solutions, where cars are evenly spaced and move with the same speed, and traveling-wave solutions, which correspond to traffic jams \cite{MarschlerSIAM}. To illustrate traffic jam solutions, it is convenient to monitor the \textit{headways} of cars: as indicated in Figure~\ref{trafficSim}, localized traveling-wave profiles correspond to traffic jams.

In \cite{MarschlerSIAM}, equation-free methods were used to trace out the bifurcation diagram for the existence and stability of traffic jam solutions as the parameter $v_0$ varies. For this analysis, the standard deviation $\sigma$ of the headways, defined by 
\[
\sigma = \sqrt{\frac{1}{N-1} \sum_{n=1}^N \left( \Delta x_n - \langle \Delta x_n \rangle \right)^2 },
\]
was used as the macroscopic variable, where $\langle \Delta x_n \rangle  = \frac{L}{N}$ is the average headway. Low standard deviations correspond to free-flow solutions, while large values of the standard deviations correspond to traffic jams (see Figure~\ref{trafficSim}). With $\sigma$ as the macroscopic variable, the lifting and restriction operators are easy to define in terms of $\sigma$: as illustrated in Figure~\ref{trafficEqFree}, lifting is accomplished by changing the locations of cars to match a given value of $\sigma$, and restricting is achieved by calculating $\sigma$ for a given profile. \revised{Note that many different car arrangements will lead to the same value of $\sigma$, and the lifting operator is therefore not uniquely defined. Choosing $\sigma$ as the macroscopic variable requires knowledge of the underlying structure of the traffic system, and our goal is to provide algorithms that identify macroscopic variables automatically from the data set.}

\begin{figure}
\centering
\includegraphics{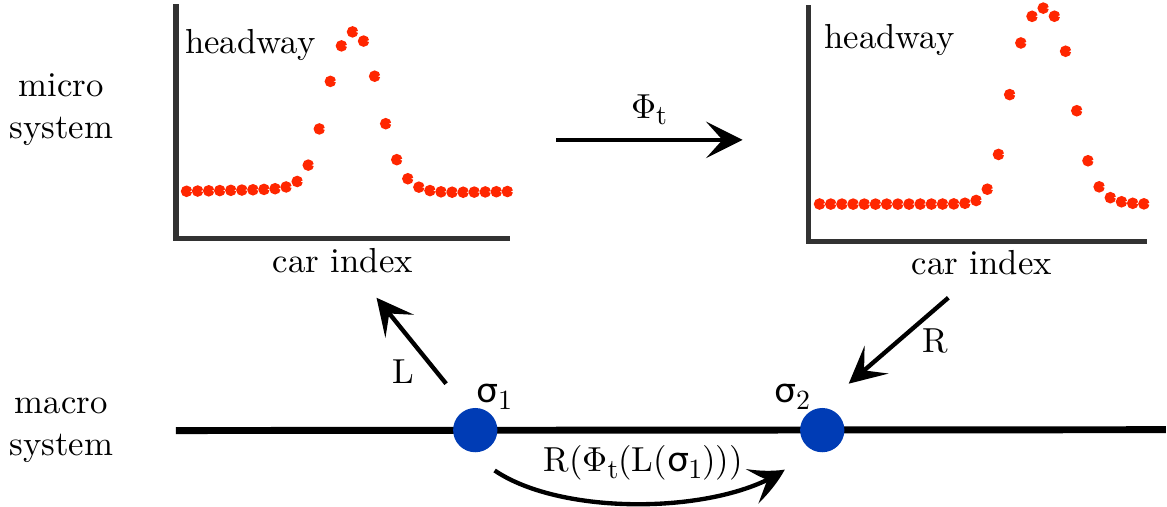}
\caption{Overview of the time stepping scheme in the context of the traffic model \cite{MarschlerSIAM}. The macroscopic state $\sigma_1$ is lifted to the microscopic profile $\mathcal{L}( \sigma_1 )$. The microsystem is then evolved for time $t$ to the next microstate $\Phi_t(\mathcal{L}(\sigma_1))$. Finally, this profile is restricted to find the evolution of the macroscopic state $\sigma_2 = \mathcal{R}(\Phi_t(\mathcal{L}(\sigma_1)))$.} \label{trafficEqFree}
\end{figure}

We will see that the proposed approach via diffusion maps will, when applied to the same traffic model, generate a parametrization that recovers the standard deviation as one of the macroscopic variables; in addition, it identifies  the location of the traveling wave as a second dimension in the parametrization. Using this new embedding as our macroscopic variables, we apply equation-free methods to reproduce the bifurcation diagram in \cite{MarschlerSIAM}, but we do so with restriction and lifting operators that emerge automatically from our diffusion map analysis without using prior knowledge of the system. In particular, we use the Nystr\"om extension for our restriction operator, which gives estimates for new components of an eigenvector of a matrix constructed from data \cite{Liu, MarschlerSIAM, Sonday2009}. We define a new lifting operator that creates microstates from linear combinations of existing data points. We apply these techniques to trace out the bifurcation diagram of traveling waves in the traffic system.

The remainder of this paper is organized as follows. In \S\ref{Equation-Free Modeling}, we present an overview of equation-free methodologies. Then, in \S\ref{Diffusion Maps}, we introduce the concept of diffusion maps and define diffusion map based operators to be used in equation-free modeling. We then apply these techniques to conduct bifurcation analysis in a traffic flow model in \S\ref{application}. Finally, in \S\ref{conclusions},  we summarize our conclusions and give an outlook of open problems. 


\section{Overview of equation-free modeling}
\label{Equation-Free Modeling}

The equation-free approach is appropriate when working with a dynamical system of large dimension $N\gg1$ that reflects a known microscopic evolution law with $N$ variables and an attracting, low-dimensional, transversely stable manifold of dimension $D$. The $N$-dimensional system is referred to as the \textit{microsystem}, and the $D$-dimensional manifold is the \textit{macrosystem}. We assume that the system exhibits a sufficiently prominent time-scale separation: more precisely, we assume that the dynamics on the $D$-dimensional attracting manifold is slow compared to the fast transverse attraction towards this $D$-dimensional slow manifold \cite{Brunovsky1, Jones, Kuehn}. Once a system is known to be slow-fast, the goal is to choose macro-level variables that parametrize the slow manifold as best as possible. The process of equation-free modeling uses two operators: lifting and restriction. 

\begin{figure}
\centering
\includegraphics{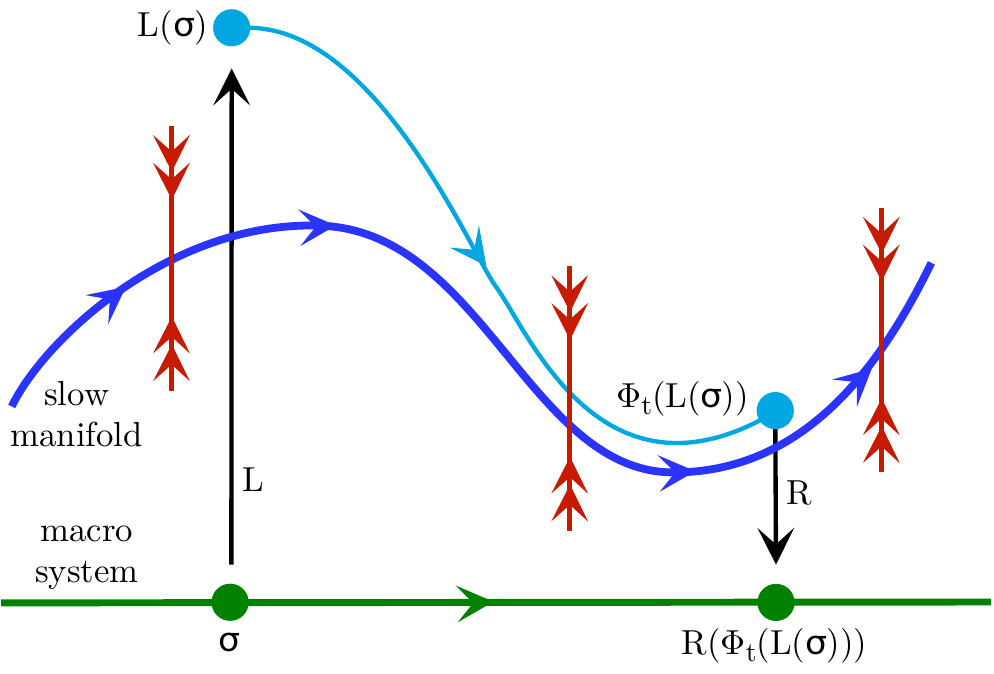}
\caption{Sketch showing the equation-free modeling approach applied to a slow-fast system. An initial macrostate $\sigma$ is lifted to a point that may possibly not lie on the slow manifold. Evolving this state for a short time will bring the profile close to the manifold. The evolved microstate can then be restricted back to the macro level or further evolved in time along the slow manifold.}
\label{manifold}
\end{figure}

The lifting operator $\mathcal{L}: \mathbb{R}^D \to \mathbb{R}^N$ maps a given \textit{macrostate} to a corresponding \textit{microstate}. Ideally, the lifting operator maps onto the slow manifold, but this is not easy to accomplish directly since the slow manifold may not be known. However, exploiting both time-scale separation and the assumption that the slow manifold is attracting, we only need to evolve the lifted microstate for a short time duration, using the time evolution $\Phi_t: \mathbb{R}^N \to \mathbb{R}^N$, to guarantee that the profile is close to the slow manifold, and can then use the resulting microstate as the image of the macrostate under lifting; see Figure~\ref{manifold} for an illustration. The additional short time evolution is often referred to as the \textit{healing step}. \revised{We note that the choice of the lifting operator may affect the required time duration of the healing step, and we refer to \cite[End of \S~II.A, and references therein]{siettos2003coarse} for a discussion of this issue.}

The restriction operator $\mathcal{R}: \mathbb{R}^N \to \mathbb{R}^D$ takes a \textit{microstate} and maps it to a low-dimensional \textit{macrostate}. Once the macrolevel parametrization is known, the restriction operator is usually much easier to define. In the traffic model, for instance, the restriction operator is defined to be the standard deviation of the headways of a microstate \cite{MarschlerSIAM}. For consistency, we require that $\mathcal{R} \circ \mathcal{L} = \mathbb{I}_{\mathbb{R}^D}$, where $\mathbb{I}_{\mathbb{R}^D}$ is the identity map for the macrolevel variables.

The equation-free framework has many benefits as an approach to study macrolevel behavior. Once the operators are defined, traditional numerical analyses of the macrosystem can be conducted without constantly simulating the microsystem \cite{Kevrekidis2009}. Since equilibria can exist in the macrostate without existing in the microstate, equation-free methods can even carry out macro-level bifurcation analyses that would be impossible with only the microsystem \cite{Kevrekidis2009}, and we will demonstrate this in \S\ref{application}.


\section{Lifting and restriction operators via diffusion maps}
\label{Diffusion Maps}

We first review diffusion maps \cite{Coifman2005, MarschlerPRE} and then use this approach to construct lifting and restriction operators from a given data set. The goal of diffusion maps is to embed a large data set in a high-dimensional space $\mathbb{R}^N$ into a low-dimensional space $\mathbb{R}^D$ with $N\gg D$ so that the local geometry of the data is preserved.


\subsection{Diffusion maps}
\label{diffusionMapMethod}

Given a high-dimensional data set $X = \{X_m \in \mathbb{R}^N \ \vert \ m = 1,\dots,M\}$ of $m$ observations, we first calculate the pairwise distances between the data points $X_m$. Although any metric can be used, we use the Euclidean norm to define $d_{ij} = \norm{X_i - X_j}$. Next, we define an affinity matrix $\mathcal{D}\in\mathbb{R}^{M\times M}$ such that a smaller distance corresponds to a high affinity and a larger distance corresponds to a small affinity. We use a Gaussian kernel to construct $\mathcal{D}$ from the pairwise distances via
\[
\mathcal{D}_{ij} = \exp\left(\frac{-d_{ij}^2}{\epsilon^2}\right) = \exp\left(\frac{-\norm{X_i - X_j}^2}{\epsilon^2}\right) 
\]
where the parameter $\epsilon$ should be chosen so that it reflects the spatial distance on which we want to resolve geometric features of the data set. Choosing $\epsilon$ too small treats all data points as singletons; picking an $\epsilon$ that is too large ignores the differences between data points: either way, we lose all geometric information. In our application in \S\ref{application}, we select
\[
\epsilon = 5\, \mbox{median} (d_{ij})_{i>j}
\]
to be five times the median of the pairwise distances $d_{ij}$, which yielded good results. Other strategies for selecting $\epsilon$ are discussed in \cite{MarschlerPRE}. We will discuss at the end of this section how we can measure the effectiveness of a given choice of $\epsilon$ for reducing the dimension quantitatively.

Next, we convert the affinity matrix $\mathcal{D}$ into a Markov transition matrix $\mathcal{M}\in\mathbb{R}^{M\times M}$ by normalizing each row via
\[
\mathcal{M}_{ij} := \frac{\mathcal{D}_{ij}}{\displaystyle \sum_{m=1}^M \mathcal{D}_{im}}.
\]
For each fixed $0<D<M$ and each choice of $D$ eigenvalues $\lambda_1,\dots,\lambda_D$ with associated eigenvectors $\psi_1,\dots,\psi_D\in\mathbb{R}^M$ of the matrix $\mathcal{M}$, we follow \cite{Erban, Laing} and map the data set $X$ into $\mathbb{R}^D$ via
\begin{equation}
X_m \longmapsto (\psi_{1,m},\dots,\psi_{D,m}) \in \mathbb{R}^D, \quad m=1,\ldots,M,
\label{diffmapembed}
\end{equation}
where $\psi_{\ell,m}$ denotes the $m^\mathrm{th}$ element of the eigenvector $\psi_\ell\in\mathbb{R}^M$ for $\ell=1,\ldots,D$.

The final step is to select the finite set of eigenvectors to represent the data set. A common choice is to choose the eigenvectors that belong to the $D$ largest eigenvalues of $\mathcal{M}$, where $D$ is chosen, for instance, to indicate a gap in the eigenvalues. This approach ignores the fact that not all eigenvectors add significantly new geometric information. Hence, we instead follow the algorithm proposed in \cite{Dsilva} to determine the dimension $D$ and the eigenvectors that provide an optimal embedding. The idea is to pick eigenvectors recursively and use local linear fits with the previously selected set of eigenvectors to see whether the new eigenvector adds sufficient information to be included. Assume that we selected the first $j-1$ eigenvectors and set $\Psi_{j-1,m}:=[\psi_{1,m},\dots,\psi_{j-1,m}]^T\in\mathbb{R}^{j-1}$ for $m=1,\dots,M$. Let $\psi_j$ be the eigenvector of $\mathcal{M}$ with the largest eigenvalue that we have not considered yet. We then compute the local fit parameters
\[
(\alpha_{j,m}, \beta_{j,m}) := \argmin_{\alpha\in\mathbb{R}, \beta\in\mathbb{R}^{j-1}}
\sum_{i \neq m} \exp\left(\frac{-\|\Psi_{j-1,m}-\Psi_{j-1,i}\|^2}{\epsilon^2}\right)
\left(\psi_{j,i} - \left(\alpha + \beta^T \Psi_{j-1,i}\right) \right)^2
\]
where ``local" refers to data points whose Gaussian distance is small, and the accompanying cross-validation error for the linear fit given by
\begin{equation}\label{rj}
r_j := \sqrt{\frac{\sum_{m=1}^M (\psi_{j,m} - (\alpha_{j,m}+\beta_{j,m}^T \Psi_{j-1,m}))^2}{\sum_{m=1}^M \psi_{j,m}^2}}.
\end{equation}
Note that small values of $r_j$ indicate that $\psi_j$ is locally well approximated by $\psi_1,\ldots,\psi_{j-1}$, so that including $\psi_j$ will not improve the embedding. We therefore include only eigenvectors with large $r_j$ values in our diffusion-map embedding and stop its recursive definition once the values of $r_j$ stay close to zero. A good choice of  $\epsilon$ will result in a steep transition of the sequence $r_j$ from values close to one to values close to zero.


\subsection{Lifting and restriction operators}
\label{liftrestrict}

In general, equation-free modeling requires the definition of lifting and restriction operators that depend on the specific system we want to solve. Previous approaches rely on in-depth understanding of the relationship between the microscopic and macroscopic states. Here, we provide an algorithm for the construction of lifting and restriction operators that depends only on a given set of data points or observations and on a given diffusion-map embedding.


\paragraph{Restriction.}

First, we describe how we construct the restriction operator based on a given data set and the accompanying diffusion-map embedding. Assume $X\in\mathbb{R}^{N\times M}$ is an existing data set, and $\mathcal{M}\in\mathbb{R}^{M\times M}$ denotes the associated Markov transition matrix with eigenvalues $\lambda_j$ and eigenvectors $\psi_j$ for $j=1,\ldots,M$. Assume also that we picked an embedding dimension $D$ and that we ordered the eigenvalues $\lambda_j$ and eigenvectors $\psi_j$ so that the restriction operator $\mathcal{R}$ evaluated on a point $X_m\in\mathbb{R}^N$ in the data set is defined by
\begin{equation}\label{restrict}
\mathcal{R}(X_m) := \left( \psi_{1,m},\ldots,\psi_{D,m} \right)\in\mathbb{R}^D, \qquad m=1,\ldots,M,
\end{equation}
where $\psi_{\ell,m}$ denotes the $m^\mathrm{th}$ component of $\psi_\ell$. We need to extend the definition of $\mathcal{R}$ so that $\mathcal{R}(X_\mathrm{new})$ is defined for each $X_\mathrm{new}\in\mathbb{R}^N$. One option is to add the new data point $X_\mathrm{new}$ to the existing data set and recalculate for the new $(M+1)$-dimensional data set, but this is cumbersome and very expensive. Instead, we follow \cite{Coifman2008} and use the Nystr\"om extension to extend $\mathcal{R}$ to new data points. This technique takes advantage of the fact that eigenvectors and eigenvalues are related by $\mathcal{M}\psi_\ell=\lambda_\ell\psi_\ell$ or, equivalently,
\begin{equation}
\psi_{\ell,m} = \frac{1}{\lambda_\ell} \sum_{j=1}^M \mathcal{M}_{m,j} \psi_{\ell,j}, \qquad \ell,m=1,\ldots,M.
\label{eigRelation}
\end{equation}
The embedding for a new data point $X_\mathrm{new}$ cannot be calculated directly from (\ref{eigRelation}), but we can modify this equation as follows to approximate the embedding. As in \S\ref{diffusionMapMethod}, we define the Gaussian kernel
\[
\mathcal{D}_{\mathrm{new},m} = \exp\left(\frac{-\norm{X_\mathrm{new} - X_m}^2}{\epsilon^2}\right)
\]
and use this expression to define
\[
\psi_{\ell,\mathrm{new}} := \frac{1}{\lambda_\ell} \sum_{m=1}^M 
\frac{\mathcal{D}_{\mathrm{new},m}}{\sum_{j=1}^M \mathcal{D}_{\mathrm{new},j}} \psi_{\ell,m}.
\]
Following \cite{Liu, MarschlerPRE, Sonday2009}, we then set
\begin{equation}\label{Rnew}
\mathcal{R}(X_\mathrm{new}) := \left(\psi_{1,\mathrm{new}}, \dots, \psi_{D,\mathrm{new}} \right),
\end{equation}
which extends the definition of the restriction operator $\mathcal{R}$ to include the new data point $X_\mathrm{new}$.


\paragraph{Lifting.}

Next, we focus on the lifting operator. Given a macrostate in $\phi\in\mathbb{R}^D$, we need to define a lifted microstate $X=\mathcal{L}(\phi)\in\mathbb{R}^N$ so that $\mathcal{R}(X)=\mathcal{R}(\mathcal{L}(\phi))\in\mathbb{R}^D$ is close to $\phi$. Our goal is to construct a lifting operator using only the given data set in $\mathbb{R}^N$ and the parametrization via diffusion maps, \revised{which we accomplish by solving an optimization problem to interpolate between the original data points. Earlier work in \cite{Erban,Laing,Sonday2009} approached this problem using simulated annealing, which is computationally more expensive. Other complementary approaches to the construction of lifting operators are discussed in \cite[\S3.2]{chiavazzo2014reduced}.}

To set up the algorithm, we choose an integer $K$ with $K\geq D+1$. Given the data points $X_m\in\mathbb{R}^N$ with $m=1,\ldots,M$, we define
\[
\phi_m := \mathcal{R}(X_m) = \left( \psi_{1,m},\ldots,\psi_{D,m} \right)\in\mathbb{R}^D, \qquad m=1,\ldots,M.
\]
Given a macrostate $\phi_\mathrm{target}\in\mathbb{R}^D$, we first find the $K$ macrostates $\phi_{m_k}$ with $k=1,\ldots,K$ that are closest to $\phi_\mathrm{target}$ in $\mathbb{R}^D$. We then define the lifted state to be
\[
\mathcal{L}(\phi_\mathrm{target}) := \sum_{k=1}^K a_k X_{m_k},
\]
where the coefficient vector $(a_1,\ldots,a_K)\in\mathbb{R}^K$ is determined as the solution to the optimization problem
\begin{equation}\label{argmin}
(a_1,\ldots,a_K) := \argmin_{(b_1,\dots,b_K)\in [0,1]^K}
\left\{ \norm{\phi_\mathrm{target} - \mathcal{R}\left(\sum_{k=1}^K b_k X_{m_k}\right)} \mbox{ subject to }
\sum_{k=1}^K b_k = 1 \right\}.
\end{equation}
Thus, we define the lifted microstate to be the element in the convex hull of the $K$ microstates whose restriction is closest to the specified targeted macrostate. Note that (\ref{argmin}) will always have a solution since $\mathcal{R}$ is continuous and the domain is compact. In general, (\ref{argmin}) may have multiple solutions. In practice, choosing larger values of $K$ ensures that zero is achieved as a minimum and evolving solutions forward will bring them close to the underlying attracting manifold: as discussed in the next section, we did not encounter any difficulties with potential discontinuities of the lifting operators during arclength continuation.

\begin{figure}
    \centering
    \includegraphics[width=0.5\textwidth]{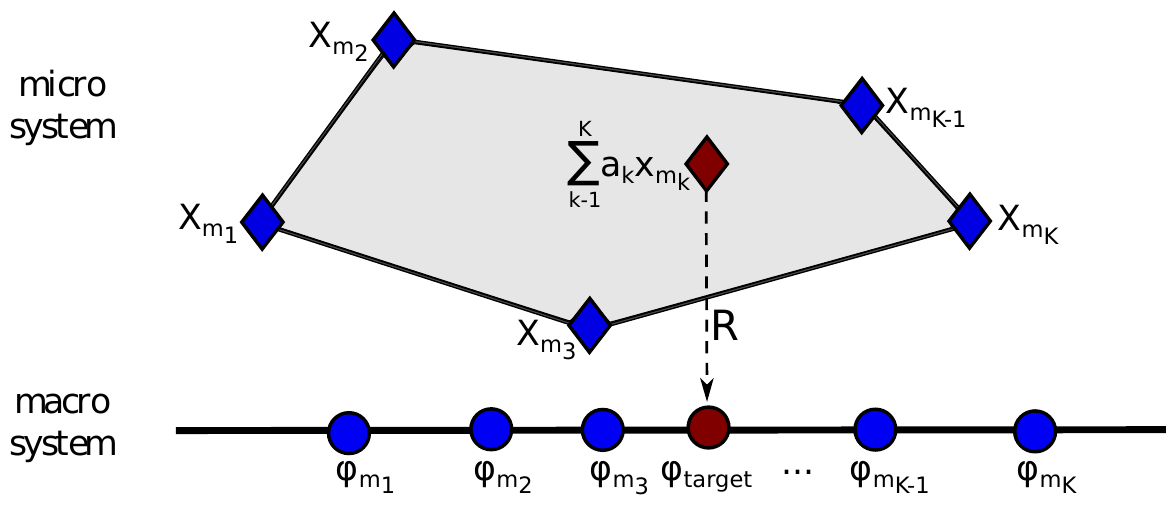}
    \caption{Shown is a visualization of the lifting operator. We solve for a linear combination of the $K$ microstates corresponding to the $K$ nearest macrostates such that the linear combination restricts to $\phi_{\text{target}}$.  }
    \label{fig:my_label}
\end{figure}


\section{Case study: Traffic model}
\label{application} 

In this section, we will use the traffic model introduced in \S\ref{introduction} to, first, demonstrate the accuracy and efficiency of the lifting and restriction operators we defined in \S\ref{liftrestrict} and, second, use these operators to compute bifurcation diagrams using equation-free modeling.

\subsection{Traffic model}
\label{trafficModel}

We write the traffic model introduced in \S\ref{introduction} as the first-order system
\begin{equation}
\label{trafficSys}
    \begin{aligned}
        \frac{\rmd x_n}{\rmd t} &= y_n, \\
        \frac{\rmd y_n}{\rmd t} &= \frac{1}{\tau}\left[ V(x_{n+1} - x_n) - y_n \right], \qquad n=1,\ldots,N
    \end{aligned}
\end{equation}
with $x_n\in\mathbb{R}/L\mathbb{Z}$, where the velocity function $V(d)$ is given by
\begin{equation}\label{e:ov}
V(d) = v_0(\tanh(d-h) + \tanh(h)).
\end{equation}
We note that (\ref{trafficSys}) is posed on the $2N$-dimensional space $(\mathbb{R}/L\mathbb{Z} \times \mathbb{R})^{2N}$, and we denote the solution of (\ref{trafficSys}) with initial condition $P=(x_n,y_n)_{n=1,\dots,N}$ evaluated at time $t$ by $\Phi_t(P)$. We record that (\ref{trafficSys}) respects the action of the discrete symmetry group $\mathbb{Z}_{N}$ given by
\begin{equation}\label{symmetry}
\left(\mathbb{R}/L\mathbb{Z} \times \mathbb{R}\right)^{2N} \longrightarrow
\left(\mathbb{R}/L\mathbb{Z} \times \mathbb{R}\right)^{2N}, \quad
(x_n, y_n)_{n=1,\ldots,N} \longmapsto (x_{(n+l)\mathrm{mod} N}, y_{(n+l)\mathrm{mod} N})_{n=1,\ldots,N}
\end{equation}
for $l\in\mathbb{Z}_N$, which corresponds to relabeling the cars consecutively. Throughout, we will fix the parameters as in Table~\ref{tab:param}, and focus on the emergence of free-flow and traffic-jam solutions as the parameter $v_0$, which appears in the velocity function, varies.

\begin{table}
    \centering
    \caption{Parameter descriptions and values}
    \label{tab:param}
    \begin{tabular}{llll}
        \hline
        Parameter & Description & Value & Used in \\
        \hline
        $N$ & number of cars & 30 & (\ref{trafficSys}) \\
        $L$ & length of the road & 60  & (\ref{trafficSys}) \\
        $\tau^{-1}$ & inertia of the car & 1.7  & (\ref{trafficSys}) \\
        $h$ & desired safety distance between cars & 2.4  & (\ref{e:ov}) \\ \hline
        $v_0$ & optimal velocity parameter & Uniform([0.96, 1.1])  & \S4.3 \\
        $A$ & amplitude of initial conditions in the dataset & Uniform([0, 4.5])  & \S4.3 \\
        $t_\mathrm{stop}$ & time evolved to create dataset & Exp(mean=$700$, shift=$200$)  & \S4.3 \\ \hline
        $K$ & number of points used for lifting & $3 \ (n=1)$  & \S4.4 \\
            & & $8 \  (n=2)$  & \S4.4 \\ \hline
        $t_\mathrm{skip}$ & healing evolution time & 300 & \S4.5 \\
        $\delta$ & time evolved for finite difference approximation & 240 & \S4.5  \\
        $s$ & continuation step size & 0.0025 ($m=5,000$)  & \S4.5 \\
            & & 0.01 ($m=1,000$) & \S4.5 \\
        $\nu$ & integer multiple of period sought & 7 & \S4.5 \\
        \hline
    \end{tabular}
\end{table}

The free-flow solution of the microsystem (\ref{trafficSys}) is defined by
\[
x_n(t) = (n-1)\frac{L}{N} + t V\left(\frac{L}{N}\right) \mod L, \qquad
y_n(t) = V\left(\frac{L}{N}\right), \qquad n=1,\ldots,N.
\]
It captures traffic flows where all cars keep the same distance $L/N$ from each other and travel with the velocity dictated by the constant headway. Traffic jams are captured by the traveling-wave ansatz
\[
(x_n(t), y_n(t)) = (x_*(n-ct), y_*(n-ct)), \qquad n=1,\ldots,N,
\]
where $(x_*(\xi), y_*(\xi))$ are $N$-periodic functions that describe the profile of the traveling wave, and $c$ is its speed. Substituting this ansatz into (\ref{trafficSys}) shows that the $N$-periodic profile $(x_*(\xi), y_*(\xi))$ and the wave speed $c$ need to satisfy the system
\begin{equation}
    \begin{aligned}
        -c \frac{\rmd x_*}{\rmd\xi}(\xi) &= y_*(\xi)\\
        -c \frac{\rmd y_*}{\rmd\xi}(\xi) &= \frac{1}{\tau} \left(V(x_*(\xi + 1) - x_*(\xi)) - y_*(\xi)\right)
    \end{aligned}
    \label{waveSoln}
\end{equation}
of delay differential equations. Note that the wave speed $c$ is a variable that needs to be solved for as part of (\ref{waveSoln}).


\subsection{Computing bifurcation diagrams using the microsystem}
\label{cont} 

\begin{figure}
\centering
\includegraphics[width=0.75\textwidth]{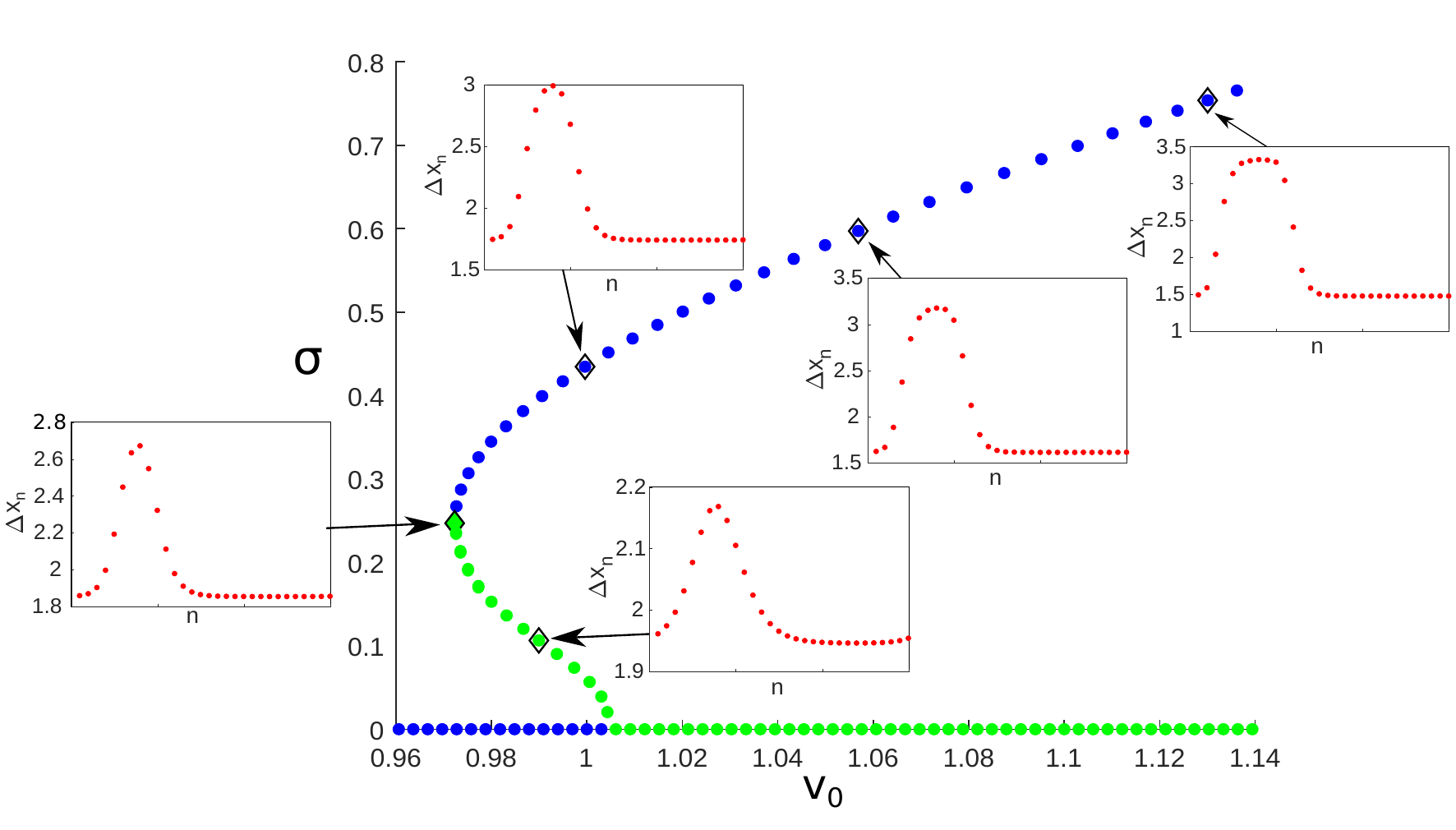}
\caption{Shown is the zero set of the function $\mathcal{F}^\mathrm{tw}$ obtained by pseudo-arclength continuation. The associated traveling-wave profiles are shown for selected points (marked with black diamonds) on the bifurcation branch. Note that the scale on the vertical axes of the insets varies to better illustrate the shape of the profiles. The algorithm detects a fold point at $(v_0, \sigma) \approx (0.97, 0.25)$, where stability changes: blue dots mark stable states, and green dots mark unstable states.}
\label{microBifProfiles}
\end{figure}

\begin{figure}
\centering
\includegraphics[width=0.5\textwidth]{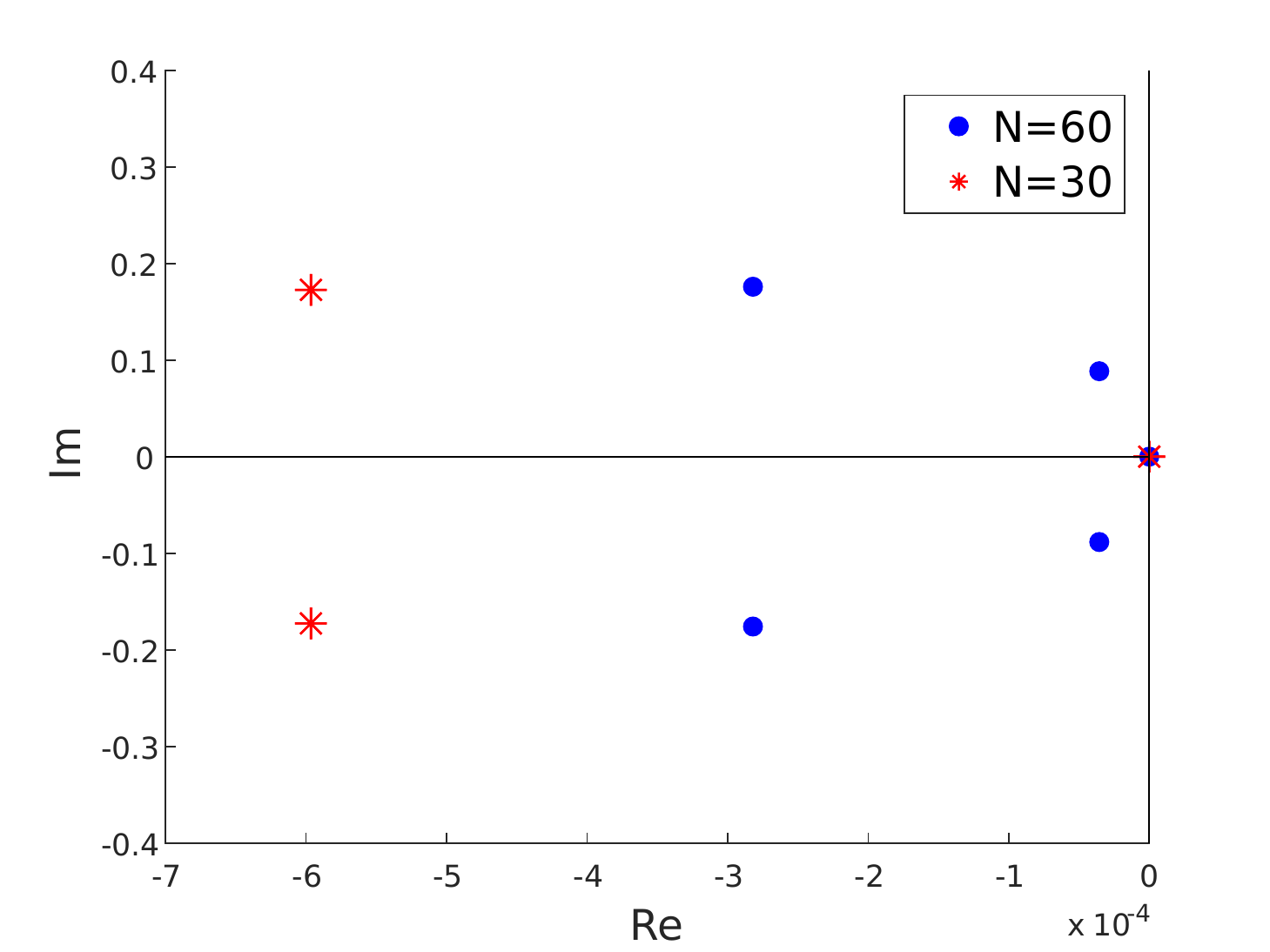}
\caption{Shown are the Floquet spectra of the linearization of the traffic system (\ref{trafficEq}) about a traveling-wave solution for the fold points for $N=30$ and $N=60$ ($v_0 = 0.97, 0.88$ and $h=2.4, 1.2$, respectively). Note that the nonzero Floquet exponents are much closer to zero for $N=60$ than for $N=30$, which indicates that the slow-fast time-scale separation becomes less pronounced as $N$ increases.}
\label{spectra}
\end{figure}

First, we use pseudo-arclength continuation to compute and continue traveling waves, and their wave speeds, as $N$-periodic solutions to (\ref{waveSoln}) as the parameter $v_0$ is varied. We will use the headways $u(\xi):=x(\xi+1)-x(\xi)$ instead of the positions $x(\xi)$ as variables. We will argue that $N$-periodic traveling waves correspond to regular roots of the function
\begin{eqnarray}\label{etw}
\mathcal{F}^\mathrm{tw}: &&
C^2(\mathbb{R}/N\mathbb{Z}) \times \mathbb{R}^2 \longrightarrow C^0(\mathbb{R}/N\mathbb{Z}) \times \mathbb{R}^2 \\ && \nonumber
(u,c,d) \longmapsto
\begin{pmatrix}
\xi \mapsto c^2 \tau \frac{\rmd^2u}{\rmd\xi^2}(\xi) - c \frac{\rmd u}{\rmd\xi}(\xi) - V(u(\xi+1)) + V(u(\xi)) + d \\
L - \sum_{n=0}^{N-1} u(n) \\
\int_0^N \left\langle \frac{\rmd u}{\rmd\xi}(\xi), u_*(\xi) - u(\xi) \right\rangle \rmd\xi
\end{pmatrix}
\end{eqnarray}
for each fixed \revised{value of the constant $v_0$ that appears in the function $V(u)$ defined in (\ref{e:ov})}. The first component $\mathcal{F}^\mathrm{tw}_1$ is the delay-differential equation (\ref{waveSoln}) written as a function of the headways; the additional term $d$ accounts for the fact that the first component (with $d=0$) has mass zero, so that $\int_0^N \mathcal{F}^\mathrm{tw}_1(u,c,d)(\xi)\rmd\xi=0$ for all $(u,c,d)$. The second component of $\mathcal{F}^\mathrm{tw}$ ensures that the headways add up to the length $L$ of the ring road. Finally, the last component is a phase condition that selects a unique profile amongst the family of spatial translates of a given solution $u_*(\xi)$; during continuation, $u_*(\xi)$ is normally taken to be the solution obtained at a previous continuation step. The following result gives conditions that guarantee that the set of roots of $\mathcal{F}^\mathrm{tw}$ consists of regular zeros.

\begin{theorem}\label{t1}
Fix $v_0$. If (i) $\mathcal{F}^\mathrm{tw}(u_*,c_*,0)=0$, (ii) the null space of $\rmD_u\mathcal{F}^\mathrm{tw}_1(u_*,c_*,0)$ is two-dimensional and spanned by $\frac{\rmd u_*}{\rmd\xi}, v\in C^2(\mathbb{R}/N\mathbb{Z})$ with $\sum_{j=0}^{N-1}v(n)\neq0$, and (iii) $\rmD_c\mathcal{F}^\mathrm{tw}_1(u_*,c_*,0)$ is not in the range of $\rmD_u\mathcal{F}^\mathrm{tw}_1(u_*,c_*,0)$, then $\rmD_{(u,c,d)}\mathcal{F}^\mathrm{tw}(u_*,c_*,0)$ has a bounded inverse.
\end{theorem}

\begin{proof}
We give only a brief outline of the proof. The key observations are that $\rmD_u\mathcal{F}^\mathrm{tw}_1$ is Fredholm of index zero and that elements in its range have mass zero. Since $\frac{\rmd u_*}{\rmd\xi}$ is contained in the null space of $\rmD_u\mathcal{F}^\mathrm{tw}_1$, it is not difficult to show that the null space is at least two-dimensional, and we assumed that its dimension is indeed two and that the null space is spanned by $\frac{\rmd u_*}{\rmd\xi}$ and $v$. Using this information and the remaining assumptions, it is now straightforward to prove that the null space of the full linearization $\rmD_{(u,c,d)}\mathcal{F}^\mathrm{tw}_1(u_*,c_*,0)$ is trivial, which proves the theorem since this operator is also Fredholm with index zero by the bordering lemma (see \cite[Lemma~2.3]{Beyn}). 
\end{proof}

Theorem~\ref{t1} indicates that we can use arclength continuation with a secant predictor and Newton's method as corrector to compute branches of traveling waves by solving $\mathcal{F}^\mathrm{tw}(u,c,d;v_0)=0$, where $\mathcal{F}^\mathrm{tw}$ depends on $v_0$ through the optimal velocity function $V(u)=V(u;v_0)$. We implemented this algorithm in Fourier space to take advantage of spectral convergence.

The result of the numerical continuation is visualized in Figure~\ref{microBifProfiles} using the standard deviation $\sigma$ of the headways $u_*(n):=x_*(n+1)-x_*(n)$. Figure~\ref{spectra} shows the Floquet spectra of the linearization of the microscale system about sample traveling waves for $N=30$ and $N=60$. In both cases, $\lambda=0$ is an eigenvalue of multiplicity two. We observe that the gap between the nonzero Floquet exponents and the eigenvalues at the origin is much smaller for $N=60$ than for $N=30$. In fact, the gap will shrink to zero as $N$ goes to infinity due to the presence of a conservation law in the continuum limit and, in particular, the slow-fast time-scale separation will become less prominent as $N$ increases. For this reason, we focus on the case $N=30$ in the remainder of this paper.


\subsection{Constructing embeddings using diffusion maps}
\label{s:embedding}

Our goal is to use diffusion maps to construct an embedding of the essential dynamics of the microsystem (\ref{trafficSys}) into a low-dimensional space and identify macroscopic variables that parametrize the reduced dynamics.

\paragraph{Construction of the data set.}

We construct the data set $X$ to which we apply the diffusion-map approach as follows. We generate $m=5000$ initial conditions of the form
\[
\label{trafficInit}
    x_n(0) = \frac{L(n-1)}{N} + A \sin\left(\frac{2\pi n}{N}\right), \quad
    y_n(0) = V\left(\frac{L}{N}\right), \qquad n=1,\ldots,N,
\]
and draw values for the coefficient $A\in[0, 4.5]$ and the parameter $v_0\in[0.96,1.1]$ randomly using the uniform distribution on these intervals. Each corresponding trajectory of (\ref{trafficSys}) is evolved using an ODE solver until a stopping time $t_\mathrm{stop}$ is reached that is drawn from a shifted exponential curve with mean 700 and shift 200. The collection of the $M$ end states of the headways in $\mathbb{R}^N$ defines the data set $X$.  We emphasize that we do not include information about the velocities in our data set. By varying the parameters and the time captured, our data set $X$ is comprised of varying wave shapes as well as some free-flow profiles. Figure~\ref{fig:data} shows that this sampling results in good coverage of the space encapsulating the bifurcation diagram. 

\begin{figure}
    \centering
    \includegraphics[width=0.48\textwidth]{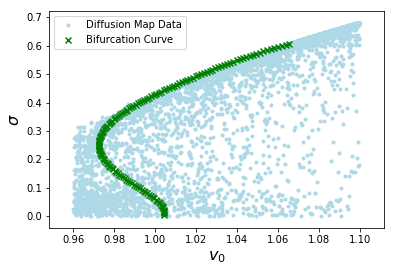}
    \includegraphics[width=0.48\textwidth]{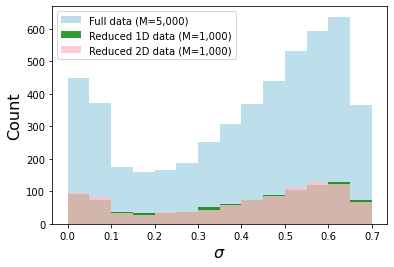}
    \caption{Shown are the data points generated for the diffusion map.}
    \label{fig:data}
\end{figure}

\paragraph{Reduction to one dimension: factoring out the discrete symmetry.}

\begin{figure}
    \centering
    \includegraphics[width=0.48\textwidth]{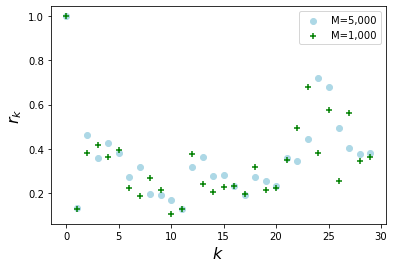}
    \includegraphics[width=0.48\textwidth]{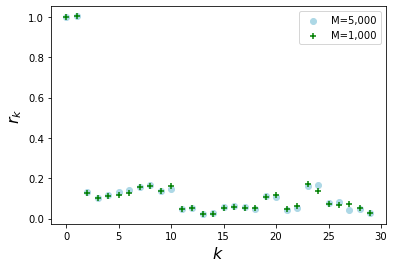}
    \caption{Shown are the local linear fit coefficients $r_k$ \revised{defined in (\ref{rj})} as functions of the index $k$ for the diffusion maps for $D=1$ (left panel) and $D=2$ (right panel).}
    \label{fig:r}
\end{figure}

First, we explicitly factor out the discrete symmetry (\ref{symmetry}) present in the model (\ref{trafficSys}). We accomplish this by shifting the indices in each profile in our data set using the symmetry (\ref{symmetry}) so that the maximum headway $\max(x_{n+1}-x_n)_{n=1,\ldots,N}$ inside the profile is achieved at $n=10$. If all solutions converge to, or at least resemble, traveling waves, factoring out the discrete symmetry effectively factors out the one-dimensional phase of all traveling waves and should therefore reduce the effective dimension of the embedding by one. Applying the diffusion-map approach outlined in \S\ref{diffusionMapMethod} to the set $X_\mathrm{align}$ of aligned elements in $\mathbb{R}^N$ \revised{and computing the linear fit coefficients defined in (\ref{rj})}, we indeed find that we can take $D=1$ as the embedding dimension so that the resulting restriction operator (\ref{restrict}) maps the aligned data set $X_\mathrm{align}$ into $\mathbb{R}$; see Figure~\ref{fig:r} (left panel). 

\revised{Figure~\ref{sigmaPlots} shows that the diffusion-map variable embedding data set into $\mathbb{R}$ is linearly related to the standard deviation of the headways. In particular, our diffusion-map approach automatically generates the parametrization introduced previously in \cite{MarschlerSIAM} based on \emph{a~priori} knowledge of the dynamics. In general, we might expect that the diffusion-map variables are mapped one-to-one to a coordinate system defined by physically relevant variables, and we refer to \cite[Inset in Figure~5]{Frewen}, \cite[Figures~6-7]{Kattis}, \cite[Figure~15]{Rajendran}, and \cite[\S~IV.A and Figure~5]{Sonday2009} for other examples where similar relationships were observed.}

\begin{figure}
\centering
\includegraphics[width=0.3\textwidth]{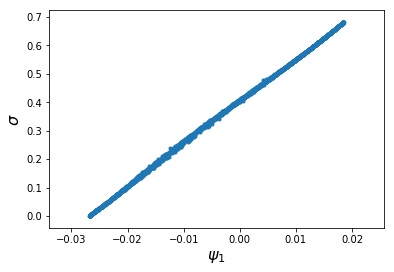}
\caption{Plotted are the values of the restriction operator $\mathcal{R}(X_m)=\psi_{1,m}$ against the standard deviation $\sigma$ evaluated at the data points $X_m$ for $m=1,\ldots,M$. Since there is a \revised{one-to-one} relationship between the diffusion map embedding and standard deviation of each data point, the two represent the same feature.}
\label{sigmaPlots}
\end{figure}

\begin{figure}
\centering
\includegraphics[width=\textwidth]{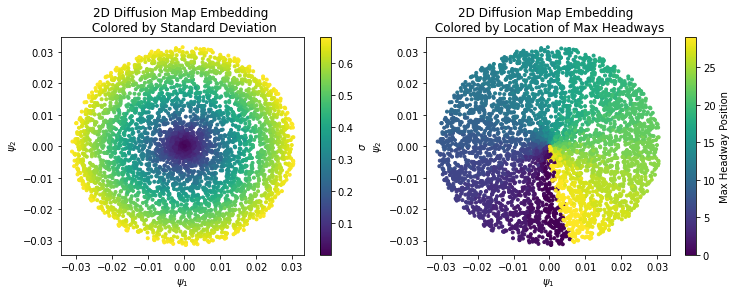}
\caption{Shown is the image of the data set $X$ under the diffusion-map embedding into $\mathbb{R}^2$ where the colors indicate the standard deviation (left panel) of the pre-images $X_m\in X$ and the wave peak position (right panel).}
\label{donut}
\end{figure}

\paragraph{Reduction to two dimensions.}
Next, we apply the diffusion-map approach from \S\ref{diffusionMapMethod} directly to the original data set $X$ without any alignment or other adjustments. In this case, we can take $D=2$ as the embedding dimension, and the resulting restriction operator (\ref{restrict}) therefore maps the data set $X$ into $\mathbb{R}^2$; \revised{see Figure~\ref{fig:r} (right panel)}. Figure~\ref{donut} shows that the planar embedding parametrizes the data set through polar coordinates where the radial direction corresponds to the standard deviation of headways and the angular direction captures the location of the peak of solutions along the circular ring road.

\paragraph{Reducing the number of data points to $M=1000$.}
To see if we can produce the same results with fewer data points, we also reduce the original diffusion map from $M=5000$ to $M=1000$. We use the $5000$ point diffusion map to inform the down-sampling of the data. For the one-dimensional case, we sort the embedding to be numerically ascending, uniformly sample the  $M=1000$ data points corresponding to those embeddings, and then recompute the diffusion map on those points. Since the two-dimensional diffusion map resembles a disc, we covert the embedding to polar coordinates,  uniformly sample $2000$ points radially, and then further uniformly sample $m=1000$ points with respect to the angular component. Finally, we recompute the diffusion map using the data points corresponding to the sampled embeddings.  Figure~\ref{fig:data} shows that this down-sampling approach preserves the distribution of $\sigma$ values in the dataset. We also observe the same relationship with $\sigma$ and the location of the maximum headways in the embeddings. 


\subsection{Lifting and restriction operators}

We now test the accuracy of the lifting and restriction operators that we defined in \S\ref{liftrestrict} based on the diffusion-map embedding constructed in \S\ref{s:embedding}.

\begin{figure}
\centering
\includegraphics[width=\textwidth]{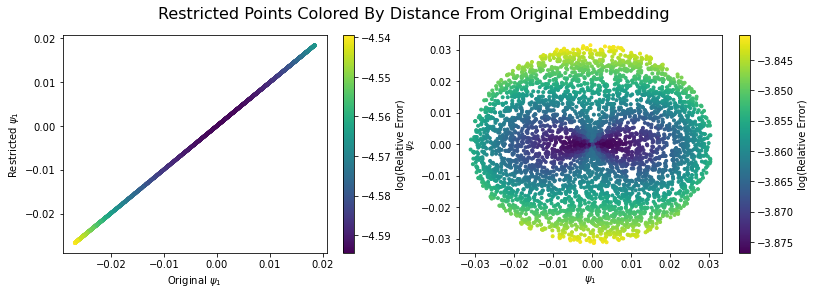}
\caption{To test the accuracy of the restriction operator defined via the Nystr\"om extension, we visualize the differences $\mathcal{R}(X_m)-\mathcal{R}_{X\setminus\{X_m\}}(X_m)$ (see main text for their definitions). Left panel: For $D=1$, we plot the one-dimensional coordinates of $\mathcal{R}(X_m)$ and $\mathcal{R}_{X\setminus\{X_m\}}(X_m)$ against each other as $m$ varies and indicate the logarithm of the relative error of their difference using colors. Right panel: For $D=2$, we plot the images $\mathcal{R}(X_m)$ as $m$ varies and visualize the logarithm of the relative error of their difference using colors.}
\label{restrictTests}
\end{figure}

First, we test the accuracy of the restriction operator defined via the Nystr\"om extension. For each fixed element $X_m\in\mathbb{R}^{N}$ of our data set $X$ (or $X_\mathrm{align}$), we first calculate the image $\mathcal{R}(X_m)$ of $X_m$ under the restriction operator (\ref{restrict}). Separately, we compute the restriction operator $\mathcal{R}_{X\setminus\{X_m\}}$ by applying the algorithm outlined in \S\ref{liftrestrict} to the data set $X\setminus\{X_m\}$, obtained from $X$ by removing the element $X_m$, and apply this restriction operator to the removed element $X_m$ to obtain $\mathcal{R}_{X\setminus\{X_m\}}(X_m)$ via the Nystr\"om extension (\ref{Rnew}). The norm of difference $\mathcal{R}(X_m)-\mathcal{R}_{X\setminus\{X_m\}}(X_m)$ measures the accuracy of our approach, and Figure~\ref{restrictTests} shows that the magnitude of the difference of these two images for both one- and two-dimensional embeddings is less than $10^{-5}$. For $D=1$, the average relative error is about $0.26\%$, while it is about $1.38\%$ for $D=2$. Plotting the errors as function of the reduced macrosystem, we see that the error is smallest in the center of domain; see Figure~\ref{restrictTests}. Since the Nystr\"om extension essentially takes a linear combination of the existing points, this observation is not unexpected. 

\begin{figure}
\centering
\includegraphics[width=\textwidth]{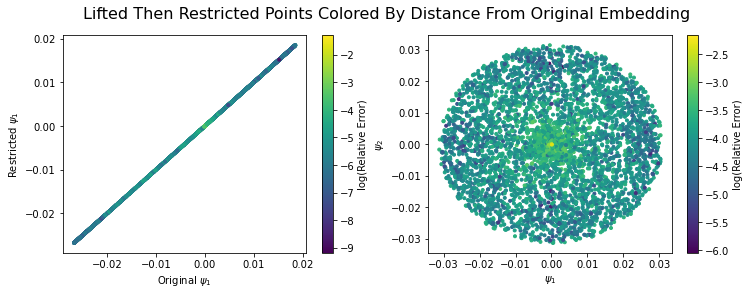}
\caption{Shown are the norms of the difference $\mathbb{I}_{\mathbb{R}^D}-\mathcal{R}\circ\mathcal{L}$ evaluated on the image $\mathcal{R}(X_\mathrm{align})$ for $D=1$ (left panel) and on $\mathcal{R}(X)$ for $D=2$ (right panel).}
\label{liftTests}
\end{figure}

Next, the theoretical approach to equation-free modeling assumes that $\mathcal{R}\circ\mathcal{L}=\mathbb{I}_{\mathbb{R}^D}$ is the identity in the macrovariables in $\mathbb{R}^D$. To test this property, we calculate the sup-norm of the map
\[
\mathbb{I}_{\mathbb{R}^D}-\mathcal{R}\circ\mathcal{L}: \quad
\mathcal{R}(Y)\subset\mathbb{R}^D \longrightarrow \mathbb{R}^D, \quad
\phi \longmapsto \phi - \mathcal{R}(\mathcal{L}(\phi))
\]
for $Y=X,X_\mathrm{align}$ and plot the results in Figure~\ref{liftTests} separately for the aligned and the original data sets. In the lifting operator, we set the number of interpolating points to $K=3$ for $D=1$ and $K=8$ for $D=2$. We observe very little difference in accuracy based on the value of $K$ but chose these values as they show the greatest accuracy. For $D=1$, the average relative error is less than $0.27\%$. For $D=2$, this error increases to $1.63\%$. We observe the least accuracy near the embeddings corresponding to low values of $\sigma$, likely due to the fact that these traffic jam solutions are unstable and thus have more variation in sampled profiles. 

\subsection{Computing bifurcation diagrams using the macrosystem}

In the preceding sections, we presented a data-driven approach to constructing lifting and restriction operators based on embeddings derived from diffusion maps and validated this approximation. In this section, we use an equation-free modeling approach based on these operators to compute the bifurcation diagram of traveling-wave solutions of (\ref{trafficSys}) in the reduced $n$-dimensional embedding space separately for $D=1$ and $D=2$. 

\begin{figure}
    \centering
    \includegraphics[width = 0.5\textwidth]{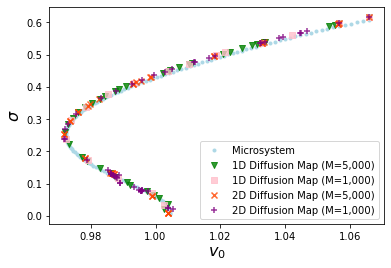}
    \caption{Shown is the bifurcation diagram in $\sigma$ coordinates computed for the microsystem directly, the one-dimensional, two-dimensional, full, and reduced diffusion maps. }
    \label{fig:sigmas}
\end{figure}

\paragraph{One-dimensional reduction: continuing fixed points.}

\begin{figure}
\centering
\includegraphics[width=0.48\textwidth]{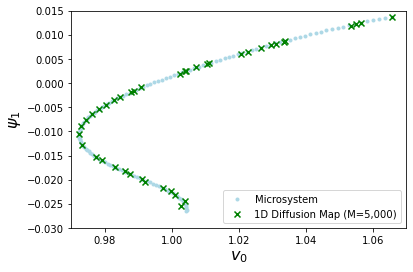}
\includegraphics[width=0.48\textwidth]{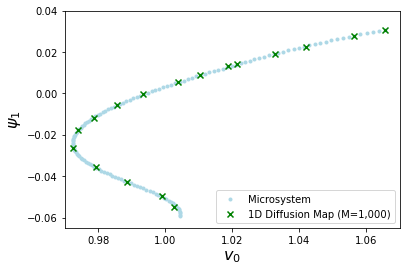}
\caption{Shown is a comparison of the bifurcation diagram computed using the equation-free approach for the reduced system with $D=1$ associated with the aligned data set and the diagram computed using pseudo-arclength continuation for the microsystem. Left: $M=5,000$ points, Right: $M=1,000$ points}
\label{bifurcation}
\end{figure}

First, we focus on the lifting and restriction operators constructed from the aligned data set $X_\mathrm{align}$, where we explicitly factored out the discrete symmetry. As shown above, the restriction operator maps into $\mathbb{R}$, and the parametrization of the macrosystem corresponds to the standard deviation of the headways. Since the phase is effectively factored out, we focus on computing and continuing equilibria of the macrosystem defined implicitly using an equation-free model.

We denote by $\Phi_t(P;v_0)$ the solution of the microsystem (\ref{trafficSys}) with parameter value $v_0$ that belongs to the initial condition $P=(x_n,y_n)_{n=1,\ldots,N}\in\mathbb{R}^N$. The equation-free macrosystem is then defined by the finite-difference quotient
\[
\frac{\rmd\phi}{\rmd t} = F(\phi,v_0) := 
\frac{\mathcal{R}\left(\Phi_{t_\mathrm{skip}+\delta}(\mathcal{L}(\phi);v_0)\right) - \mathcal{R}\left(\Phi_{t_\mathrm{skip}}(\mathcal{L}(\phi);v_0)\right)}{\delta},
\]
where $\phi\in\mathbb{R}$ denotes the macrovariable, and the time steps $t_\mathrm{skip},\delta>0$ are chose to obtain time-scale separation in the dynamics. We then compute and continue roots of the function $F(\phi,v_0)$ using pseudo-arclength continuation with a secant predictor of step size $s$ and a Newton corrector in the space $(\phi,v_0)\in\mathbb{R}^2$.

The results are shown in $\sigma$ coordinates in Figure~\ref{fig:sigmas} and the diffusion map coordinates Figure~\ref{bifurcation} for both the full diffusion map and the reduced diffusion map. Both bifurcation diagrams resemble the diagram computed previously in \cite{MarschlerSIAM} and the diagram in Figure~\ref{microBifProfiles} computed using continuation in the microsystem. We notice that the reduced diffusion map is more robust to continuation parameter choices, and we hypothesize that this difference results from the artificial alignment of the profiles in the one-dimensional diffusion map. In the $M=5,000$ diffusion map, aligning the data introduces noise around each value of $\sigma$ since there will be many profiles originating from different positions to potentially lift. Downsampling the diffusion map to $M=1,000$ reduces some of that noise, thus resulting in a more robust method.

\paragraph{Two-dimensional reduction: continuing periodic orbits.}

\begin{figure}
\centering
\includegraphics[width=\textwidth]{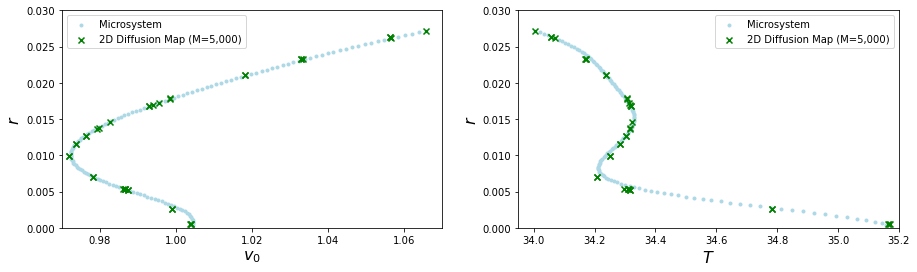}
\caption{Shown is a comparison of the bifurcation diagram of periodic orbits computed using the equation-free approach for the reduced system with $D=2$ associated with the data set $X$ and the diagram computed using pseudo-arclength continuation for the microsystem. The left and right panel show projections into $(v_0,r)$ and $(T, r)$, respectively.}
\label{2D}
\end{figure}

Next, we focus on the lifting and restriction operators constructed from the original data set $X$, which results in a two-dimensional macrosystem given by
\begin{equation}\label{efm}
\frac{\rmd\phi}{\rmd t} = F(\phi,v_0) := 
\frac{\mathcal{R}\left(\Phi_{t_\mathrm{skip}+\delta}(\mathcal{L}(\phi);v_0)\right) - \mathcal{R}\left(\Phi_{t_\mathrm{skip}}(\mathcal{L}(\phi);v_0)\right)}{\delta},
\end{equation}
where $\phi\in\mathbb{R}^2$ again denotes the, now two-dimensional, macrovariable. Traveling waves of the microsystem (\ref{trafficSys}) correspond to periodic orbits of the planar system (\ref{efm}). We denote by $\tilde{\Phi}_t(\phi;v_0)$ the solution operator of the macrosystem (\ref{efm}) and define the one-dimensional Poincare section $\mathbb{R}\phi_\mathrm{PS}$ for a nonzero vector $\phi_\mathrm{PS}\in\mathbb{R}^2$ in the macrosystem. Periodic orbits with period $T$ of (\ref{efm}) can then be found as solutions of the system $\mathcal{F}(r,T,v_0)=0$ given by
\[
\mathcal{F}:\quad
\mathbb{R}^3 \longrightarrow \mathbb{R}^2, \quad
(r, T, v_0) \longmapsto \tilde{\Phi}_{t_\mathrm{skip} }(r\phi_\mathrm{PS}; v_0) - \tilde{\Phi}_{t_\mathrm{skip} + \nu T}(r\phi_\mathrm{PS}; v_0).
\]
We solve this system for $(r,T,v_0)$ using again pseudo-arclength continuation with a secant predictor and a Newton corrector. Since the periods are roughly $ T \approx 34.5$, we choose to evolve for $\nu = 7$ periods to match $\delta = 240 \approx \nu T $ used in the one-dimensional system. 

The results are illustrated in $\sigma$ coordinates in Figure~\ref{fig:sigmas} and in diffusion map coordinates in Figure~\ref{2D}. The full diffusion map diagram is computed accurately in $(r, T, v_0)$-space and deviates from the diagram obtained from the microsystem only slightly near the fold point. The reduced diffusion map diagram shown in Figure~\ref{2D_reduced} is less accurate, particularly on the stable branch. In this case, downsampling the data likely left gaps in the $ (r, T, v_0) $ space of interest, leading to less accurate results.

\begin{figure}
\centering
\includegraphics[width=\textwidth]{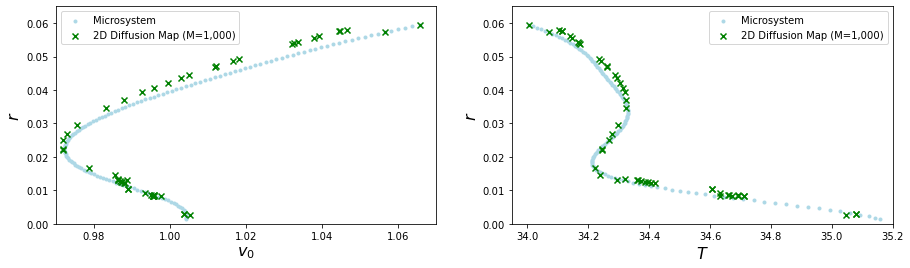}
\caption{Shown is a comparison of the bifurcation diagram of periodic orbits computed using the equation-free approach for the reduced system with $D=2$ associated with the reduced data set $X$ and the diagram computed using pseudo-arclength continuation for the microsystem. The left and right panel show projections into $(v_0,r)$ and $(T, r)$, respectively.}
\label{2D_reduced}
\end{figure}


\section{Conclusions}
\label{conclusions}

We considered large-dimensional dynamical systems, referred to as the microsystem, with a slow-fast time-scale separation. Equation-free modeling attempts to reduce the dynamics of the microsystem to an implicitly defined, lower-dimensional macrosystem that captures the dynamics on the slow manifold of the microsystem corresponding to the slow time scale. The macrosystem can then be used, for instance, to carry out direct numerical simulations with larger step sizes compared to simulations of the microsystem or to compute and continue stationary solutions or periodic orbits of the macrosystem that may correspond to more complex patterns in the microscopic variables. In previous applications, the macroscopic variables were identified based on insights into the microsystem: for instance, in the context of traffic-flow models, it makes sense to use the standard deviation from the free-flow solution to capture traffic-jam solutions.

For this paper, our goal was to develop an application-independent approach to equation-free modelling that does not rely on being able to make an explicit ansatz for the macroscopic variables. We focused on a data-driven approach and used diffusion maps to embed the data set into a lower-dimensional space, identify macroscopic variables that parametrize the low-dimensional space, and construct lifting an restriction operators that connect the micro- and macroscopic systems. Our case study demonstrated that these operators can be used to continue traffic-flow patterns as steady states or as periodic orbits in the macroscopic system.

It would be interesting to see whether this approach can be used to compute and continue more complex patterns that cannot be characterized directly in the microsystem: examples are patches of turbulent fluid surrounded by laminar flow, for instance, or other spatially chaotic structures whose overall shape that may be parametrized by appropriate macroscopic variables.

In our application, we found that the most important aspect was appropriately sampling data to build the diffusion map so that all regions of interest in phase space are well-covered.  Another avenue for future research is to find a better method for identifying the underlying data set, for instance by finding better sampling techniques for the initial data and the stopping times.

\paragraph{Acknowledgments.}
Tracy Chin, Jacob Ruth, and Rebecca Santorella were supported by the NSF grant DMS-1439786  through the Summer@ICERM program.
Rebecca Santorella was also supported by the NSF through grant 1644760.
Bjorn Sandstede was supported by the NSF under grants DMS-1408742, DMS-1714429, and CCF-1740741.


\bibliographystyle{plain}
\bibliography{EquationFreeModeling}

\end{document}